\documentclass[11pt]{article}%
\usepackage{amssymb,amsmath,amsfonts,amsthm,array,bm,color}%
\usepackage{amsmath}%
\usepackage{amsfonts}%
\usepackage{amssymb}%
\usepackage{graphicx}
\providecommand{\U}[1]{\protect\rule{.1in}{.1in}}

\numberwithin{equation}{section}

\newtheorem{theorem}{Theorem}[section]
\newtheorem{lemma}[theorem]{Lemma}
\newtheorem{corollary}[theorem]{Corollary}
\newtheorem{proposition}[theorem]{Proposition}
\newtheorem{remark}[theorem]{Remark}

\def\<{\langle}
\def\>{\rangle}
\def\d{{\rm d}}
\def\L{\mathcal{L}}

\def\E{\mathbb{E}}

\def\N{\mathbb{N}}
\def\P{\mathbb{P}}
\def\R{\mathbb{R}}
\def\T{\mathbb{T}}
\def\Z{\mathbb{Z}}

\def\eps{\varepsilon}

\begin{document}

\title{Mean field limit of point vortices with environmental noises to deterministic 2D Navier-Stokes equations}

\author{Franco Flandoli\footnote{Email: franco.flandoli@sns.it. Scuola Normale Superiore di Pisa, Piazza dei Cavalieri, 7, 56126 Pisa, Italy.} \quad
Dejun Luo\footnote{Email: luodj@amss.ac.cn. Key Laboratory of RCSDS, Academy of Mathematics and Systems Science, Chinese Academy of Sciences, Beijing 100190, China, and School of Mathematical Sciences, University of the Chinese Academy of Sciences, Beijing 100049, China. }}

\maketitle

\vspace{-5mm}

\begin{abstract}
We consider point vortex systems on the two dimensional torus perturbed by environmental noise. It is shown that, under a suitable scaling of the noises, weak limit points of the empirical measures are solutions to the vorticity formulation of deterministic 2D Navier-Stokes equations.
\end{abstract}



\section{Introduction}

The mean field limit is widely used to derive macroscopic PDEs from large systems of
interacting particles, as a useful method to reduce the complexity of systems.
It is often natural to consider particle
systems subjected to random perturbations, mainly by independent Brownian
motions; the limit mean field PDEs in this case are nonlinear parabolic equations,
called the McKean-Vlasov equations \cite{McKean}. The coupling method is very efficient
to treat Lipschitz continuous interaction kernels, giving rise to explicit
convergence rate of empirical measures to solutions of the mean field equation,
see the classical work of Sznitman \cite{Sznitman} and also \cite{Meleard}.
Particle systems with singular kernels have attracted a lot of attention, a notable
example being the point vortex model for 2D Euler equations with the
Biot-Savart kernel as the interaction kernel, see e.g. \cite{MP82, Schochet96}
for the deterministic case and \cite{Osada, Meleard01, FourHauMis, JabinWang} for the stochastic case.
For exchangeable systems, it is well known that the mean field limit is equivalent
to the phenomenon of propagation of chaos, cf. \cite[p.177, Proposition 2.2]{Sznitman}
and also \cite{MM13, HM14} for stronger notions of chaos.
The readers can find more detailed accounts of the literature in the introduction of
\cite{FourHauMis} and in the nice survey \cite{JabinWang17}.

The paper \cite{FourHauMis}, dealing with the mean field limit of stochastic
point vortices to the deterministic 2D Navier-Stokes equations, the noise
being additive and independent for each particle, states at page 1425 an open
problem concerning the generalization of the result to the case of
\textit{environmental noise}, which means that the same space-dependent noise
acts on all particles -- the action differing just by the position of the
particle, where the noise is evaluated. This open problem, which is at the origin of
the present work, has two possible faces. One is the convergence of the
empirical measure to a stochastic 2D Euler equation, where stochasticity
reflects the random environment, still present in the limit. This has been
done in \cite{Coghi-Fla} under Lipschitz condition on the interaction kernel;
see \cite{Coghi-Mau} for a scaling limit result on point vortices with
regularized Biot-Savart kernel and suitably chosen regularizing parameter,
and the recent preprint \cite{Rosenzweig} for a mean field limit without
smoothing the Biot-Savart kernel. Another face, the one
considered here, is to rescale the space covariance of the noise,
simultaneously with the increasing number of particles, in such a way that the
noise becomes more and more uncorrelated, going heuristically in the direction
of the independent additive noises acting on different particles. We present
here a partial solution to this second case, showing that any weak limit of the empirical measures is a
probability measure, time dependent, solution to the 2D Navier-Stokes equations.

The equation for the empirical measure of point vortices contains a martingale
which has to converge to zero in a scaling regime leading to the deterministic
2D Navier-Stokes equation in vorticity formulation. In the classical case of
independent noise on each particle, this
convergence is standard. In the case of environmental noise, it may be a
difficult problem, as it is here (cf. Proposition \ref{key-prop}). To overcome this difficulty we have
developed nontrivial estimates based on entropy, inspired by \cite{FourHauMis}.

It turns out that the entropy estimate plays also an important role in proving the convergence
of the nonlinear part in the equation for empirical measures. Indeed, using Young's inequality,
we are able to derive a uniform estimate on the expected value of the Hamiltonian of
random point vortices, see Lemma \ref{lem-Hamiltonian} below.
As in the deterministic theory (see e.g. \cite{Schochet96}), such estimate implies non-concentration
of point vortices, as well as the fact that weak limits of empirical measures are continuous measures
containing no delta Dirac mass. These results are crucial for showing the convergence of the nonlinear part
to the desired limit as the number of vortices tends to infinity.

For technical reasons, we consider point vortices on the torus $\T^2= [-1/2,1/2]^2$, endowed with periodic boundary condition. This allows the explicit choice of a family of divergence free vector fields $\{\sigma_k \}_k$ (see below) and simplifies some computations; moreover, the compactness of torus makes it easier for integrability arguments (see e.g. Proposition \ref{key-prop}). Let $K:\T^2 \to \R^2$ be the Biot-Savart kernel on the 2D torus, whose basic properties will be recalled at the beginning of Section 2. Here we just mention that $K$ is singular near the origin since $|K(x)|\sim \frac1{|x|}$ as $|x|\to 0$. We consider the following system of $N$-point vortices perturbed by multiplicative noises: for $i=1,\ldots, N$,
  \begin{equation}\label{particle-systems.0}
  \d X^{N,i}_t = \frac1N \sum_{j=1,j\neq i}^N K\big( X^{N,i}_t - X^{N,j}_t \big)\,\d t + \d W\big(t, X^{N,i}_t \big),\quad X^{N,i}_t =X^i_0,
  \end{equation}
where $\{X^i_0\}_{i\geq 1}$ is an i.i.d. sequence of random variables on $\T^2$ whose law will be specified below, and $W(t,x)$ is a space-time noise, white in time and colored in space, modelling the random environment in which the vortices evolve. Unlike in the usual particle systems where different particles are perturbed by mutually independent Brownian noises (see e.g.  \cite{Sznitman, FourHauMis, JabinWang} and the survey \cite{JabinWang17}), the noise in \eqref{particle-systems.0} is the same for each particle, that is, the random vector field $W(t,x)$.  Such noise is called an environmental noise.

Under quite general conditions on the spatial covariance function of $W(t,x)$ (cf. the second paragraph on p.107 of \cite{Kunita} or Theorem 4.2.5 therein for an abstract result), one can represent the field $W(t,x)$ as a random series. In this work, we assume that
  $$W(t,x)= \sum_{k\in \Z^2_0} \theta_k\, \sigma_k(x) W^k_t ,$$
where $\Z^2_0= \Z^2\setminus \{0\}$ is the set of nonzero integer points, $\{\sigma_k \}_{k\in \Z^2_0}$ a family of divergence free vector fields on $\T^2$ defined as in \eqref{CONS} below, $\{W^k_\cdot \}_{k\in \Z^2_0}$ a family of independent standard Brownian motions defined on some filtered probability space $(\Omega, \mathcal F, (\mathcal F_t)_{t\geq 0} ,\P)$; finally, $\theta\in \ell^2(\Z^2_0)$, the latter being the usual space of square summable real sequences indexed by $\Z^2_0$. It is enough to consider those $\theta$ with only finitely many nonzero components (see for instance the example in Remark \ref{rem-1} below), and satisfying the symmetry property:
  \begin{equation}\label{theta}
  \theta_k = \theta_j \quad \mbox{whenever } |k|= |j|.
  \end{equation}
Using these notations, the point vortex system \eqref{particle-systems.0} can be written more precisely as follows: for $i=1,\ldots, N$,
  $$\d X^{N,i}_t = \frac1N \sum_{j=1,j\neq i}^N K\big( X^{N,i}_t - X^{N,j}_t \big)\,\d t + \sum_k \theta_k\, \sigma_k\big( X^{N,i}_t \big)\, \d W^k_t,\quad X^{N,i}_t =X^i_0. $$
Before moving forward, we remark that, under suitable nondegeneracy conditions on the noise, the stochastic point vortex system is globally well posed for Lebesgue almost every initial configuration, cf. \cite[p.1456, Theorem 8]{FGP2} (note that this result does not require the bracket generating condition in Hypothesis 1 on p.1451); see also \cite[Theorem 1.2]{LS20} for a similar result on the vortex model of mSQG equations.

As mentioned above, if we fix a noise $W(t,x)$ (i.e. fix some $\theta\in \ell^2$) and consider the mean field limit of empirical measures, then the limit equation will be a stochastic PDE, cf. \cite{Coghi-Fla, Rosenzweig}. In order to get a deterministic limit equation we need to introduce a scaling parameter in the noise part. Therefore, we take a family $\{\theta^N \}_{N\in \N} \subset \ell^2$ satisfying \eqref{theta} for each $N\in\N$, and consider the point vortex system
  \begin{equation}\label{particle-systems}
  \d X^{N,i}_t = \frac1N \sum_{j=1,j\neq i}^N K\big( X^{N,i}_t - X^{N,j}_t \big)\,\d t + \eps_N \sum_{k\in \Z^2_0} \theta^N_k \sigma_k \big( X^{N,i}_t \big)\,\d W^k_t,\quad X^{N,i}_t =X^i_0
  \end{equation}
for $i=1,\ldots, N$, where ($\nu>0$ is the noise intensity)
  \begin{equation}\label{eps-N}
  \eps_N = \frac{2\sqrt{\nu}} {\|\theta^N \|_{\ell^2}}.
  \end{equation}
Such scaling of noise is motivated by recent works \cite{Gal, FGL}, where the linear transport or 2D Euler equations driven by multiplicative noise of transport type are shown to converge to deterministic parabolic equations or 2D Navier-Stokes equations (see also \cite{FlaLuo-2} where the limit equation is the 2D Navier-Stokes driven by space-time white noise).

We denote the empirical measure by
  $$S^N_t = \frac1N \sum_{i=1}^N \delta_{X^{N,i}_t} $$
and define the covariance function
  \begin{equation}\label{covariance}
  Q_N(x,y) = \sum_{k\in \Z^2_0} \big( \theta^N_k \big)^2 \sigma_k(x) \otimes \sigma_k(y),\quad x,y\in \T^2.
  \end{equation}
It can be shown that $Q_N(x,y)$ depends only on the difference $x-y$ (cf. the proof of Lemma \ref{2-lem}) and thus it will be denoted as $Q_N(x-y)$, where $Q_N$ is a $2\times 2$ matrix-valued function defined on $\T^2$. Due to the choice of $\eps_N$ in \eqref{eps-N} and equality \eqref{useful-identity} below, it holds that
  $$\eps_N^2 Q_N(0)= 2\nu I_2 \quad \mbox{for all } N\geq 1,$$
where $I_2$ is the $2\times 2$ identity matrix. Finally, let $\mathcal P(\T^2)$ be the collection of probability measures on $\T^2$ and $H^s(\T^2)\, (s\in \R)$ the usual Sobolev spaces on $\T^2$.

Our purpose is to prove

\begin{theorem}\label{main-thm}
Let $T>0$ be given. Assume that
\begin{itemize}
\item[\rm(a)] the initial data $\{ X^i_0\}_{i\in \N}$ is a sequence of i.i.d. $\mathcal F_0$-measurable random variables with law $\mu_0 = f_0\,\d x \in \mathcal P(\T^2)$ for some density function $f_0:\T^2\to \R_+$ with finite entropy;
\item[\rm(b)] the sequence $\big\{ \theta^N \big\}_{N\in \N}$ satisfies
  \begin{equation}\label{main-thm.1}
  \lim_{N\to \infty} \eps_N^2\, Q_N(x) =0 \quad \mbox{for all } x\in \T^2\setminus \{0\}.
  \end{equation}
\end{itemize}
Then the laws $\eta_N$ of $S^N_\cdot\, (N\in \N)$ are tight on $C\big([0,T], H^{-s}(\T^2) \big)$ for any $s>1$, and any weak limit of $\{\eta_N \}_N$ is supported on weak solutions of the deterministic 2D Navier-Stokes equations in vorticity form:
  \begin{equation}\label{main-thm.2}
  \partial_t \xi + (K\ast \xi)\cdot \nabla \xi = \nu \Delta \xi, \quad \xi|_{t=0} = f_0.
  \end{equation}
\end{theorem}

\begin{remark}\label{rem-1}
\begin{itemize}
\item[\rm (i)] Unfortunately, we cannot prove that the weak limits are the unique solution to \eqref{main-thm.2}, due to the lack of good estimates on the empirical measures. Indeed, we can only prove that the weak limits $\tilde \xi \in L^2\big(0,T; H^{-1}(\T^2) \big)$ almost surely, and thus the corresponding velocity $\tilde u=K\ast \tilde\xi \in L^2\big(0,T; L^2(\T^2) \big)$, cf. Corollary \ref{cor-regularity} below. In \cite[Theorem 1.5]{ChesLuo} Cheskidov and Luo proved that weak solutions in this class to the velocity form of the 2D Navier-Stokes equations are not unique. Thus the problem we leave open is an interesting one for future research.

\item[\rm (ii)] The condition \eqref{main-thm.1} is used to prove that the martingale part in \eqref{empirical-measure} (the equation for empirical measures) tends to 0 in mean square. Here is a simple example for \eqref{main-thm.1}. Let
  $$\theta^N_k = \frac1{|k|} {\bf 1}_{\{ |k|\leq N\} },\quad k\in \Z^2_0,\, N\in \N,$$
then it is clear that
  $$\eps_N^2 = 4\nu \bigg(\sum_{|k|\leq N} \frac1{|k|^2} \bigg)^{-1} \sim \frac{4\nu}{\log N}$$
and (cf. the proof of Lemma \ref{2-lem}; $k^\perp= (k_2,-k_1)$)
  $$Q_N(x) = \sum_{|k|\leq N} \frac{k^\perp \otimes k^\perp}{|k|^4} \cos(2\pi k\cdot x).$$
We know that $\lim_{N\to \infty} Q_N(x)$ exists for all $x\in \T^2\setminus \{0\}$, thus the condition \eqref{main-thm.1} holds.
\end{itemize}
\end{remark}

This paper is organized as follows. In Section 2 we first briefly recall the basic properties of the Biot-Savart kernel on $\T^2$ and define the vector fields $\{\sigma_k \}_{k\in \Z^2_0}$ used above; then, we turn to derive the equation for the empirical measures $S^N_t$, and establish a uniform estimate on the entropy of joint density functions of random point vortices \eqref{particle-systems}. The proof of Theorem \ref{main-thm} is given in Section \ref{sec-scaling-limit}, where the main difficulty is to show that the martingale parts in the equations of empirical measures vanish in the scaling limit, as well as the convergence of the nonlinear parts. The proofs rely heavily on the entropy estimate in Section \ref{subsec-entropy}.

\section{Preparations}

First, we recall some basic properties of the Biot-Savart kernel $K$ on $\T^2$. We have $K= \nabla^\perp G= (\partial_2 G, -\partial_1 G)$, where $G$ is the Green function on $\T^2$. On the whole space $\R^2$, we have the simple expression $G_{\R^2}(x)=\frac1{2\pi} \log|x|$; on $\T^2$, it is known that
  \begin{equation}\label{Green-funct}
  G(x)= \frac1{2\pi} \log|x| + r(x), \quad x\in \T^2 \setminus\{0\},
  \end{equation}
where $r$ is a smooth function on $\T^2$. By definition $K$ is smooth and divergence free away from the origin $0\in \T^2$, and $K(-x)= -K(x)$ for all $x\neq 0$; moreover, it holds that
  $$|K(x)| \sim \frac1{2\pi |x|} \quad \mbox{as } |x|\to 0. $$

Next we define the vector fields $\sigma_k,\, k\in \Z^2_0$ as follows:
  \begin{equation}\label{CONS}
  \sigma_k(x)= \frac{k^\perp}{|k|} e_k(x),\quad x\in \T^2,\, k\in \Z^2_0,
  \end{equation}
where $k^\perp= (k_2, -k_1)$ and
  $$e_k(x) = \sqrt{2}\, \begin{cases}
  \cos(2\pi k\cdot x), & k\in \Z^2_+; \\
  \sin(2\pi k\cdot x), & k\in \Z^2_-,
  \end{cases}$$
with $\Z^2_+= \{k\in \Z^2_0: (k_1>0) \mbox{ or } (k_1=0, k_2>0)\}$ and $\Z^2_- = -\Z^2_+$. Then $\{\sigma_k \}_{k\in \Z^2_0}$ is a CONS of the space of square integrable and divergence free vector fields on $\T^2$ with zero mean.

The rest of this section consists of two parts. In Section \ref{subsec-empirical} we derive the equation fulfilled by the empirical measure $S^N_t$, see \eqref{empirical-measure}. We introduce in Section \ref{subsec-entropy} the rescaled entropy functional and prove a uniform estimate for the entropy of joint density functions of point vortices. This estimate will play a crucial role in the proof of the main result.

\subsection{Equation for empirical measures} \label{subsec-empirical}

We want to find the equation satisfied by the empirical measure $S^N_t,\, t\geq 0$. Let $\phi\in C^2(\T^2)$; by \eqref{particle-systems} and the It\^o formula,
  $$\aligned
  \d \phi\big( X^{N,i}_t \big) =&\, \frac1N \sum_{j\neq i}K \big(X^{N,i}_t -X^{N,j}_t\big) \cdot \nabla\phi \big( X^{N,i}_t \big) \,\d t + \eps_N \sum_{k\in \Z^2_0} \theta^N_k (\sigma_k\cdot \nabla\phi ) \big( X^{N,i}_t \big)\,\d W^k_t \\
  &\ + \frac{\eps_N^2}2 \sum_{k\in \Z^2_0} \big( \theta^N_k \big)^2 {\rm Tr} \big[(\sigma_k \otimes \sigma_k) \nabla^2\phi \big]\big( X^{N,i}_t \big) \,\d t.
  \endaligned$$
The proof of the following key identity is similar to \cite[Lemma 2.6]{FlaLuo-2}, see also \cite[Section 2]{Gal}.

\begin{lemma}\label{2-lem}
It holds that
  \begin{equation}\label{useful-identity}
  \sum_{k\in \Z^2_0} \big( \theta^N_k \big)^2 (\sigma_k\otimes \sigma_k)(x) \equiv \frac12 \big\| \theta^N \big\|_{\ell^2}^2 I_2, \quad x\in \T^2,
  \end{equation}
where $I_2$ is the $2\times 2$ identity matrix.
\end{lemma}

\begin{proof}
We give the proof for the reader's convenience. For any $x,y\in \T^2$, using the definition of $\sigma_k$ and the fact that $\theta^N$ satisfies \eqref{theta}, we have
  $$\aligned
  Q_N(x,y)= &\ \sum_{k\in \Z^2_0} \big( \theta^N_k \big)^2 \sigma_k(x)\otimes \sigma_k(y)= \sum_{k\in \Z^2_0} \big( \theta^N_k \big)^2 \frac{k^\perp \otimes k^\perp}{|k|^2} e_k(x) e_k(y) \\
  =&\ 2\sum_{k\in \Z^2_+} \big( \theta^N_k \big)^2 \frac{k^\perp \otimes k^\perp}{|k|^2} \big[\cos(2\pi k\cdot x)\cos(2\pi k\cdot y) + \sin(2\pi k\cdot x)\sin (2\pi k\cdot y) \big] \\
  =&\ 2\sum_{k\in \Z^2_+} \big( \theta^N_k \big)^2 \frac{k^\perp \otimes k^\perp}{|k|^2} \cos(2\pi k\cdot (x-y)) =\sum_{k\in \Z^2_0} \big( \theta^N_k \big)^2 \frac{k^\perp \otimes k^\perp}{|k|^2} \cos(2\pi k\cdot (x-y)).
  \endaligned $$
Therefore, $Q_N(x,y)$ only depends on the displacement $x-y$, which, for simplicity of notation, will be denoted by $Q_N(x-y)$. In particular,
  $$Q_N(0) = \sum_{k\in \Z^2_0} \big( \theta^N_k \big)^2 \sigma_k(x)\otimes \sigma_k(x) = \sum_{k\in \Z^2_0} \big( \theta^N_k \big)^2 \frac{k^\perp \otimes k^\perp}{|k|^2} = \sum_{k\in \Z^2_0} \frac{\big( \theta^N_k \big)^2 }{|k|^2}
  \begin{pmatrix}
  k_2^2 &  -k_1 k_2 \\  -k_1 k_2 & k_1^2
  \end{pmatrix}.$$
First, we have
  $$Q_N^{1,2}(0) = - \sum_{k\in \Z^2_0} \frac{\big( \theta^N_k \big)^2 }{|k|^2} k_1 k_2 =0 $$
since, by \eqref{theta}, the sum of the four terms involving $(k_1, k_2),\, (-k_1, k_2),\, (k_1, -k_2),\, (-k_1, -k_2)$ cancel each other. Next, using again \eqref{theta},
  $$Q_N^{1,1}(0)= \sum_{k\in \Z^2_0} \frac{\big( \theta^N_k \big)^2 }{|k|^2} k_2^2 = \sum_{k\in \Z^2_0} \frac{\big( \theta^N_k \big)^2 }{|k|^2} k_1^2 = Q_N^{2,2}(0)$$
since the points $(k_1, k_2)$ and $(k_2, k_1)$ appear in pair. Therefore,
  $$ Q_N^{1,1}(0)= Q_N^{2,2}(0) = \frac12 \sum_{k\in \Z^2_0} \frac{\big( \theta^N_k \big)^2 }{|k|^2} \big( k_1^2 + k_2^2 \big) = \frac12 \sum_{k\in \Z^2_0} \big( \theta^N_k \big)^2 = \frac12 \big\| \theta^N \big\|_{\ell^2}^2.$$
This completes the proof.
\end{proof}

Hence, by \eqref{useful-identity} and the definition \eqref{eps-N} of $\eps_N$, we obtain
  $$\frac{\eps_N^2}2 \sum_{k\in \Z^2_0} \big( \theta^N_k \big)^2 {\rm Tr} \big[(\sigma_k \otimes \sigma_k) \nabla^2\phi \big]\big( X^{N,i}_t \big) = \nu \Delta \phi\big( X^{N,i}_t \big) .$$
As a result, for $1\leq i\leq N$,
  $$\aligned
  \d \phi\big( X^{N,i}_t \big) &= \frac1N \sum_{j\neq i}K \big(X^{N,i}_t -X^{N,j}_t\big) \cdot \nabla\phi \big( X^{N,i}_t \big) \,\d t + \nu \Delta \phi\big( X^{N,i}_t \big)\,\d t\\
  &\quad + \eps_N \sum_{k\in \Z^2_0} \theta^N_k (\sigma_k\cdot \nabla\phi ) \big( X^{N,i}_t \big)\,\d W^k_t.
  \endaligned $$
Denoting by $\big\< S^N_t, \phi \big\>= \frac1{N} \sum_{i=1}^N \phi\big( X^{N,i}_t \big)$, then we have
  $$\aligned
  \d \big\< S^N_t, \phi \big\> &= \frac1{N^2} \sum_{1\leq i\neq j\leq N} K \big(X^{N,i}_t -X^{N,j}_t\big) \cdot \nabla\phi \big( X^{N,i}_t \big) \,\d t + \nu \big\< S^N_t, \Delta\phi \big\>\,\d t \\
  &\quad + \eps_N \sum_{k\in \Z^2_0} \theta^N_k \big\< S^N_t, \sigma_k\cdot \nabla\phi \big\>\,\d W^k_t.
  \endaligned $$
Using the fact that $K(-x)=-K(x)$ for all $ x\in \T^2\setminus \{0\}$, we can rewrite the first term on the right hand side as $\big\< S^N_t\otimes S^N_t, H_\phi \big\>$, where
  $$H_\phi (x,y)= \frac12 K(x-y)\cdot (\nabla\phi(x)- \nabla\phi(y))$$
is a symmetric function on $\T^2\times \T^2$, with the convention that $H_\phi (x,x)=0$. We remark that $H_\phi $ is smooth off the diagonal and bounded by $C\|\nabla^2\phi \|_\infty$ for some $C>0$ independent of $\phi$. Therefore, we get the equation for the empirical measure:
  \begin{equation}\label{empirical-measure}
  \aligned
  \d \big\< S^N_t, \phi \big\>&= \big\< S^N_t\otimes S^N_t, H_\phi \big\>\,\d t + \nu \big\< S^N_t, \Delta\phi \big\>\,\d t + \eps_N \sum_{k\in \Z^2_0} \theta^N_k \big\< S^N_t, \sigma_k\cdot \nabla\phi \big\>\,\d W^k_t.
  \endaligned
  \end{equation}

\subsection{Entropy for density functions}\label{subsec-entropy}

The relative entropy $h_N(F)$ of probability density functions $F$ on $\T^{2N} :=(\T^2)^N$ is defined as
  $$h_N(F) = \frac1N \int_{\T^{2N}} F(X)\log F(X) \,\d X,$$
where $\d X = \d x_1\ldots \d x_N$ is the Lebesgue measure on $\T^{2N}$. The simple inequality $s\log s \geq s-1 \ (s\geq 0)$ implies that $h_N(F)$ is always nonnegative. As in \cite{FourHauMis}, we add the factor $1/N$ so that if $F(x_1, \ldots, x_N) = f(x_1) \ldots f(x_N)$ for some probability density $f$ on $\T^2$ with finite entropy, then one has $h_N(F) = h_1(f)$. The functional $h_N$ enjoys the following important property: if $F$ is exchangeable (i.e. $F$ is invariant under permutations of its variables) and $F^{(2)}$ is the marginal distribution of $F$ on $\T^4$, then
  \begin{equation}\label{entropy}
  h_N(F) \geq \frac{N-1}N h_2\big( F^{(2)}\big)\quad \mbox{for all } N\geq 2.
  \end{equation}

Let $F^{N}_t$ be the density function of the law on $\T^{2N}$ of the particles $\big(X^{N,1}_t, \ldots, X^{N,N}_t\big)$ associated to \eqref{particle-systems}; then $F^{N}_0(X) = f_0(x_1) \ldots f_0(x_N)$. We want to prove an estimate on $h_N\big(F^{N}_t \big)$, for which we need to introduce some notations. Define the dispersion vector fields on $\T^{2N}$:
  $$A_k(X)= (\sigma_k(x_1), \ldots, \sigma_k(x_N)), \quad X= (x_1, \ldots, x_N)\in \T^{2N},\,  k\in \Z^2_0,$$
and the drift field $A_0$ via
  $$A_0^i(X) = \frac1N \sum_{j=1,j\neq i}^N K(x_i- x_j), \quad X= (x_1, \ldots, x_N) \notin D_N, \ 1\leq i \leq N,$$
where $D_N= \{X= (x_1, \ldots, x_N)\in \T^{2N}: \exists\, i\neq j \mbox{ such that } x_i=x_j\}$ is the generalized diagonal of $\T^{2N}$. It is clear that all the vector fields $A_k\, (k\in \Z^2_0)$ are divergence free, so is $A_0$ on $D_N^c$. Denoting by $X^N_t= \big(X^{N,1}_t,\ldots, X^{N,N}_t \big),\, t\geq 0$; then the system \eqref{particle-systems} of SDEs can be simply written as
  $$\d X^N_t = A_0\big( X^N_t\big) \,\d t + \eps_N \sum_{k\in \Z^2_0} \theta^N_k A_k \big( X^N_t\big)\,\d W^k_t.$$
We remark that the equation can also be written in the Stratonovich form since $A_k \cdot \nabla_N A_k=0$ for all $k\in \Z^2_0$, where $\nabla_N= (\nabla_{x_1},\ldots, \nabla_{x_N})$ is the gradient operator on $\T^{2N}$. The associated infinitesimal generator has the form
  $$\L \Phi(X) = \frac{\eps_N^2}2 \sum_{k\in \Z^2_0} \big(\theta^N_k \big)^2 \big\<A_k, \nabla_N \<A_k, \nabla_N \Phi\>_{\R^{2N}} \big\>_{\R^{2N}} + \<A_0, \nabla_N \Phi\>_{\R^{2N}}, \quad \Phi\in C^2 \big( \T^{2N} \big). $$

\begin{lemma}\label{lem-entropy}
For all $t >0$,
  $$h_N\big( F^{N}_t \big) \leq h_N\big( F^{N}_0 \big)= h_1(f_0).$$
\end{lemma}

\begin{proof}
The proof below is a little formal, but it can be made rigorous by approximating the initial density $f_0$ and the kernel $K:\T^2 \to \R^2$ with smooth objects, cf. \cite[Sect. 4.2]{FlaLuo-1}. The density function $F^N_t$ satisfies the Fokker--Planck equation
  $$\partial_t F^N_t = \L^\ast F^N_t = \frac{\eps_N^2}2 \sum_{k\in \Z^2_0} \big(\theta^N_k \big)^2 \big\<A_k, \nabla_N \big\<A_k, \nabla_N F^N_t \big\>_{\R^{2N}} \big\>_{\R^{2N}} - \big\<A_0, \nabla_N F^N_t \big\>_{\R^{2N}}.$$
Therefore,
  $$\aligned
  \partial_t \big( F^N_t\log F^N_t \big)=&\ \big(1+ \log F^N_t \big)\partial_t F^N_t \\
  =&\ \frac{\eps_N^2}2 \sum_{k\in \Z^2_0} \big(\theta^N_k \big)^2 \big(1+ \log F^N_t \big) \big\<A_k, \nabla_N \big\<A_k, \nabla_N F^N_t \big\>_{\R^{2N}} \big\>_{\R^{2N}} \\
  &\, - \big(1+ \log F^N_t \big) \big\<A_0, \nabla_N F^N_t \big\>_{\R^{2N}}.
  \endaligned $$
Note that $\big(1+ \log F^N_t \big) \big\<A_0, \nabla_N F^N_t \big\>_{\R^{2N}} = \big\<A_0, \nabla_N \big(F^N_t \log F^N_t \big) \big\>_{\R^{2N}}$ and that all the vector fields $A_k$ and $A_0$ are divergence free. Integrating both sides of the above equation on $\T^{2N}$ and applying integration by parts yield
  \begin{equation}\label{lem-entropy.1}
  \aligned
  \frac{\d}{\d t} h_N\big( F^N_t \big) &= \frac1N \int_{\T^{2N}} \big(1+ \log F^N_t \big)\partial_t F^N_t \,\d X \\
  &= - \frac{\eps_N^2}{2N} \sum_{k\in \Z^2_0} \big(\theta^N_k \big)^2 \int_{\T^{2N}} \frac{\big\<A_k, \nabla_N F^N_t \big\>_{\R^{2N}}^2}{F^N_t}\,\d X,
  \endaligned
  \end{equation}
which immediately gives us the desired result.
\end{proof}

\begin{remark}
Unlike in \cite{FourHauMis}, we are unable to derive, from the identity \eqref{lem-entropy.1}, estimate on the Fisher information, which was used in \cite[Lemma 3.3]{FourHauMis} to show that particles are not too close to each other.
\end{remark}

\section{Scaling limit of random point vortices} \label{sec-scaling-limit}

Recall the empirical measures $\{S^N_t: t\in [0,T]\}_{N\geq 1}$ defined in Section 1.  For any $\phi\in C^\infty(\T^2)$ it is obvious that $|\<S^N_t,\phi\>| \leq \|\phi \|_\infty$, thus, using the definition of Sobolev norm in $H^{s}(\T^2)$, one can easily show that, for any $s>1$,
  \begin{equation}\label{bound-1}
  \sup_{0\leq t\leq T} \big\|S^N_t \big\|_{H^{-s}} \leq C_s <\infty \quad \P \mbox{-a.s.}
  \end{equation}
In particular, $S^N_\cdot$ has trajectories in $L^\infty\big(0,T; H^{-s}(\T^2)\big),\, s>1$.

Let $\eta_N, \, N\in \N$ be the laws of  $S^N_\cdot$; we want to show that the family $\{\eta_N\}_{N\geq 1}$ is tight on $C\big([0,T]; H^{-s}(\T^2)\big)$ for any $s>1$. By \eqref{bound-1} and Simon's compactness result (cf. \cite[p. 90, Corollary 9]{Simon}), it is sufficient to show that $\{S^N_\cdot\}_{N\geq 1}$ is bounded in probability in $W^{1/3,4}\big(0,T; H^{-\beta}(\T^2)\big)$ for some $\beta>5$. This is an immediate consequence of the fact below:
  \begin{equation}\label{bound-2}
  \sup_{N\geq 1} \E\int_0^T \!\int_0^T \frac{\|S^N_t- S^N_s\|_{H^{-\beta}}^4}{|t-s|^{7/3}} \,\d t\d s <+\infty.
  \end{equation}

\begin{lemma}\label{lem-time-difference}
There exists a constant $C= C(T,\nu)>0$ such that for any $k\in \Z^2$, it holds
  $$\E\Big(\big\<S^N_t- S^N_s, e_k\big\>^4 \Big) \leq C|k|^8 (t-s)^2, \quad 0\leq s <t\leq T. $$
\end{lemma}

\begin{proof}
By \eqref{empirical-measure}, we have
  $$ \aligned
  \big\< S^N_t -S^N_s, e_k \big\> &= \int_s^t \big\< S^N_r\otimes S^N_r, H_{e_k} \big\>\,\d r + \nu \int_s^t \big\< S^N_r, \Delta e_k \big\>\,\d r \\
  & \quad + \eps_N \sum_{l} \theta^N_l \int_s^t \big\< S^N_r, \sigma_l\cdot \nabla e_k \big\>\,\d W^l_r.
  \endaligned $$
First, by the Burkholder-Davis-Gundy inequality,
  $$\E \bigg[\Big( \eps_N \sum_{l} \theta^N_l \int_s^t \big\< S^N_r, \sigma_l\cdot \nabla e_k \big\>\,\d W^l_r \Big)^4 \bigg] \leq C \eps_N ^4\, \E \bigg[\Big( \sum_{l} \big(\theta^N_l \big)^2 \int_s^t \big\< S^N_r, \sigma_l\cdot \nabla e_k \big\>^2\,\d r \Big)^2 \bigg]. $$
Using Cauchy's inequality, we obtain
  $$\aligned
  \sum_{l} \big(\theta^N_l \big)^2 \big\< S^N_r, \sigma_l\cdot \nabla e_k \big\>^2 &\leq \sum_{l} \big(\theta^N_l \big)^2 \frac1N \sum_{i=1}^N \big[(\sigma_l\cdot \nabla e_k)\big(X^{N,i}_t\big)\big]^2 \\
  &= \frac1N \sum_{i=1}^N \frac12 \|\theta^N\|_{\ell^2}^2 \big| \nabla e_k\big(X^{N,i}_t\big) \big|^2 \leq 4\pi^2 \|\theta^N\|_{\ell^2}^2 |k|^2,
  \endaligned $$
where in the second step we have used \eqref{useful-identity}. By the definition of $\eps_N$, we arrive at
  $$\E \bigg[\Big( \eps_N \sum_{l} \theta^N_l \int_s^t \big\< S^N_r, \sigma_l\cdot \nabla e_k \big\>\,\d W^l_r \Big)^4 \bigg] \leq C \nu^2 |k|^4 (t-s)^2. $$
Combining this estimate with the following facts
  $$\big|\big\< S^N_r\otimes S^N_r, H_{e_k} \big\> \big| \leq \|H_{e_k} \|_\infty \leq C|k|^2, \quad \big|\big\< S^N_r, \Delta e_k \big\> \big| \leq \|\Delta e_k \|_\infty \leq 4\pi^2 |k|^2, $$
we can easily prove the desired estimate.
\end{proof}

Now by Cauchy's inequality and Lemma \ref{lem-time-difference},
  $$\aligned \E \big(\|S^N_t- S^N_s\|_{H^{-\beta}}^4 \big) &= \E\Bigg[ \bigg(\sum_k \frac{\< S^N_t -S^N_s, e_k \>^2}{(1+|k|^2)^\beta} \bigg)^2 \Bigg] \\
  &\leq \bigg(\sum_k \frac{1}{(1+|k|^2)^\beta} \bigg) \bigg(\sum_k \frac{\E \big(\< S^N_t -S^N_s, e_k \>^4 \big)}{(1+|k|^2)^\beta} \bigg)\\
  &\leq C_\beta \sum_k \frac{C|k|^8 (t-s)^2 }{(1+|k|^2)^\beta} \leq C'_\beta (t-s)^2 ,
  \endaligned $$
where the last inequality is due to $\beta>5$. From this result we immediately get \eqref{bound-2}.

Summarizing the above discussions, we deduce that $\{\eta_N\}_{N\geq 1}$ is tight on $C\big([0,T]; H^{-s}(\T^2)\big)$ for any $s>1$. Therefore, Prohorov's theorem (see \cite[p.59, Theorem 5.1]{Billingsley}) implies the existence of a subsequence $\{\eta_{N_i}\}_{i\geq 1}$ converging weakly to some probability measure $\eta$ supported on $C\big([0,T]; H^{-s}(\T^2)\big)$. By Skorokhod's representation theorem (see \cite[p.70, Theorem 6.7]{Billingsley}), there exist a new probability space $\big(\tilde\Omega, \tilde{\mathcal F},\tilde \P \big)$, a sequence of random variables $\big\{\tilde S^{N_i} \big\}_{i\geq 1}$ and a random variable $\tilde\xi$ defined on $\big( \tilde\Omega, \tilde{\mathcal F},\tilde \P \big)$, such that
\begin{itemize}
\item[(i)] $\tilde\xi $ has law $\eta$ and $\tilde S^{N_i} $ has law $\eta_{N_i}\, (i\geq 1)$;
\item[(ii)] $\tilde\P$-a.s., $\tilde S^{N_i}_\cdot$ converges to $\tilde\xi_\cdot $ in the topology of $C\big([0,T]; H^{-s}(\T^2)\big)$.
\end{itemize}

\begin{remark}
Notice that, for $\tilde\P$-a.s. $\tilde\omega\in \tilde\Omega$,  $\big\{\tilde S^{N_i}_\cdot (\tilde\omega) \big\}_{i\geq 1}$ is a sequence of functions with values in the space $\mathcal P(\T^2)$ of probability measures; moreover, the above arguments show that they are equi-continuous in time in some negative Sobolev space. Therefore, up to a further subsequence, $\tilde S^{N_i}_\cdot(\tilde\omega)$ converges weakly-$\ast$ in the space of functions with values in $\mathcal P(\T^2)$; see the third paragraph in \cite[p.915]{Schochet96} for similar remarks. As a result, for $\tilde\P$-a.s. $\tilde\omega\in \tilde\Omega$ and a.e. $t\in (0,T)$, one has $\tilde\xi_t(\tilde \omega) \in \mathcal P(\T^2)$.
\end{remark}

By assertion (i), for any $i\geq 1$, $\tilde S^{N_i}_\cdot$ fulfills an equation as \eqref{empirical-measure}; therefore, for any $\phi\in C^\infty(\T^2)$, for any $t\in [0,T]$,
  \begin{equation}\label{eq-new}
  \aligned
  \big\< \tilde S^{N_i}_t, \phi \big\>&= \big\< \tilde S^{N_i}_0, \phi \big\> + \int_0^t \big\< \tilde S^{N_i}_s\otimes \tilde S^{N_i}_s, H_\phi \big\>\,\d s + \nu \int_0^t \big\< \tilde S^{N_i}_s, \Delta\phi \big\>\,\d s \\
  &\quad + \eps_{N_i} \sum_{k\in \Z^2_0} \theta^{N_i}_k \int_0^t \big\< \tilde S^{N_i}_s, \sigma_k\cdot \nabla\phi \big\>\,\d \tilde W^k_s,
  \endaligned
  \end{equation}
where $\big\{\tilde W^k \big\}_{k\in \Z^2_0}$ is a family of independent standard Brownian motions on $\big( \tilde\Omega, \tilde{\mathcal F},\tilde \P \big)$. It remains to let $i\to \infty$ in the above equation and prove that the limit $\tilde\xi$ fulfills the weak vorticity form of the deterministic 2D Navier-Stokes equation. For this purpose, it is sufficient to show the convergence of the nonlinear part and the martingale part. In the following, for simplicity of notations, we omit the tilde over $\tilde\xi,\, \tilde S^{N_i}$ and $\tilde W^k$, and write $N$ instead of $N_i$.

We first deal with the martingale part:
  \begin{equation}\label{martingale}
  M^{\phi, N}_t :=\eps_N \sum_{k\in \Z^2_0} \theta^N_k \int_0^t \big\< S^N_s, \sigma_k\cdot \nabla\phi \big\>\,\d W^k_s.
  \end{equation}

\begin{proposition}\label{key-prop}
Under the condition \eqref{main-thm.1} it holds that
  $$ \lim_{N\to \infty} \E\bigg[\sup_{t\in [0,T]} \big|M^{\phi, N}_t\big|^2 \bigg] =0.$$
\end{proposition}

\begin{proof}
We have
  $$\E\bigg[\sup_{t\in [0,T]} \big|M^{\phi, N}_t\big|^2 \bigg] \leq C \E\Big[\big|M^{\phi, N}_T\big|^2 \Big] = C \eps_N^2 \sum_{k\in \Z^2_0} \big(\theta^N_k \big)^2 \, \E \int_0^T \big\< S^N_t, \sigma_k\cdot \nabla\phi \big\>^2 \,\d t.$$
Noticing that
  $$\big\< S^N_t, \sigma_k\cdot \nabla\phi \big\>^2 = \frac1{N^2} \sum_{i,j=1}^N (\sigma_k\cdot \nabla\phi)\big( X^{N,i}_t \big) \, (\sigma_k\cdot \nabla\phi)\big( X^{N,j}_t \big),$$
the right hand side of the above inequality can be decomposed as the sum of the following two terms:
  $$\aligned
  I_1 & = C \frac{\eps_N^2}{N^2} \sum_{i=1}^N \sum_{k\in \Z^2_0} \big(\theta^N_k \big)^2 \, \E \int_0^T \big[ (\sigma_k\cdot \nabla\phi)\big( X^{N,i}_t \big) \big]^2 \,\d t, \\
  I_2&= C \frac{\eps_N^2}{N^2} \sum_{1\leq i\neq j\leq N} \sum_{k\in \Z^2_0} \big(\theta^N_k \big)^2 \, \E \int_0^T (\sigma_k\cdot \nabla\phi)\big( X^{N,i}_t \big) \, (\sigma_k\cdot \nabla\phi)\big( X^{N,j}_t \big) \,\d t.
  \endaligned$$
It follows from \eqref{useful-identity} and \eqref{eps-N} that
  \begin{equation}\label{key-prop.1}
  |I_1| = C \frac{\eps_N^2}{N^2} \sum_{i=1}^N \frac12 \big\| \theta^N \big\|_{\ell^2}^2 \, \E \int_0^T \big| \nabla\phi \big( X^{N,i}_t \big) \big|^2 \,\d t \leq \frac {CT\nu}N \|\nabla\phi\|_\infty^2 \to 0
  \end{equation}
as $N\to \infty$. Here $\|\nabla\phi\|_\infty$ is the supremum norm of $\nabla\phi(x)$ on $\T^2$.

We now turn to the second term $I_2$. By the exchangeability,
  $$\aligned
  I_2 &= C\eps_N^2 \frac{N-1}{N} \sum_{k\in \Z^2_0} \big(\theta^N_k \big)^2 \, \E \int_0^T (\sigma_k\cdot \nabla\phi)\big( X^{N,1}_t \big) \, (\sigma_k\cdot \nabla\phi)\big( X^{N,2}_t \big) \,\d t\\
  &= C\eps_N^2 \frac{N-1}{N} \sum_{k\in \Z^2_0} \big(\theta^N_k \big)^2 \int_0^T\!\!\int_{\T^4} (\sigma_k\cdot \nabla\phi)( x_1) \, (\sigma_k\cdot \nabla\phi)( x_2) F^{N,2}_t(x_1, x_2)\,\d x_1\d x_2\d t,
  \endaligned$$
where $F^{N,2}_t$ is the joint density function of $\big( X^{N,1}_t, X^{N,2}_t \big)$. We have, by the definition \eqref{covariance} of the covariance function $Q_N$,
  $$I_2 = C\eps_N^2 \frac{N-1}{N} \int_0^T\!\!\int_{\T^4} (\nabla\phi( x_1))^\ast Q_N(x_1 -x_2) \nabla\phi( x_2)\, F^{N,2}_t(x_1, x_2) \,\d x_1\d x_2\d t.$$
Therefore, for any $M>1$,
  \begin{equation}\label{key-prop.2}
  \aligned
  |I_2| \leq &\ C \|\nabla\phi\|_\infty^2\, \eps_N^2 \int_0^T\!\!\int_{\T^4} \big| Q_N(x_1 -x_2) \big| F^{N,2}_t(x_1, x_2) \,\d x_1\d x_2\d t \\
  = &\ C \|\nabla\phi\|_\infty^2\, \eps_N^2 \int_0^T\!\!\int_{\left\{ F^{N,2}_t \leq M \right\}} \big| Q_N(x_1 -x_2) \big| F^{N,2}_t(x_1, x_2) \,\d x_1\d x_2\d t \\
  & + C \|\nabla\phi\|_\infty^2\, \eps_N^2 \int_0^T\!\!\int_{\left\{ F^{N,2}_t > M \right\}} \big| Q_N(x_1 -x_2) \big| F^{N,2}_t(x_1, x_2) \,\d x_1\d x_2\d t\\
  =: & \ J_N^{(1)} + J_N^{(2)}.
  \endaligned
  \end{equation}
For the first term, one has
  $$\aligned
  J_N^{(1)} &\leq C M \|\nabla\phi\|_\infty^2 \, \eps_N^2 \int_0^T\!\!\int_{\left\{ F^{N,2}_t \leq M \right\}} \big| Q_N(x_1 -x_2) \big| \,\d x_1\d x_2\d t \\
  &\leq C M \|\nabla\phi\|_\infty^2 \, T \eps_N^2 \int_{\T^4} \big| Q_N(x_1 -x_2) \big| \,\d x_1\d x_2 .
  \endaligned$$
By the proof of Lemma \ref{2-lem}, it is easy to see that
  \begin{equation}\label{key-prop.2.5}
  \eps_N^2 \big| Q_N(x_1 -x_2) \big| \leq 2\nu \quad \mbox{for all } x_1, x_2\in \T^2.
  \end{equation}
Hence by \eqref{main-thm.1} and the dominated convergence theorem,
  \begin{equation}\label{key-prop.3}
  \lim_{N\to \infty} J_N^{(1)} =0.
  \end{equation}

Next, using again the inequality \eqref{key-prop.2.5},
  $$\aligned
  J_N^{(2)} &\leq 2\nu\, C\|\nabla\phi\|_\infty^2 \int_0^T\!\!\int_{\left\{ F^{N,2}_t > M \right\}} F^{N,2}_t(x_1, x_2) \,\d x_1\d x_2\d t \\
  &\leq \frac{C' }{\log M} \int_0^T\!\!\int_{\left\{ F^{N,2}_t > 1 \right\}} \big( F^{N,2}_t \log F^{N,2}_t \big)(x_1, x_2) \,\d x_1\d x_2\d t.
  \endaligned $$
Recalling the simple fact that $s\log s \in [-e^{-1}, 0]$ for all $s\in [0,1]$, one can easily prove
  $$\int_{\left\{ F^{N,2}_t > 1 \right\}} \big(F^{N,2}_t \log F^{N,2}_t \big)(x_1, x_2) \,\d x_1\d x_2 \leq 2 h_2\big(F^{N,2}_t\big) + e^{-1}.$$
Therefore, by \eqref{entropy} and Lemma \ref{lem-entropy} above, we obtain
  $$J_N^{(2)} \leq  \frac{C' T}{\log M} \big[ h_1(f_0) + e^{-1}\big].$$
Combining this estimate with \eqref{key-prop.2} and \eqref{key-prop.3}, we conclude that $I_2$ tends to 0 as $N\to \infty$. In view of \eqref{key-prop.1}, this completes the proof.
\end{proof}

Next we turn to prove the convergence of the nonlinear part in \eqref{eq-new}. For this purpose, we make some preparations by introducing the Hamiltonian of point vortices: for $X=(x_1,\ldots, x_N)\in \T^{2N}$,
  $$\mathcal H_N(X)= \frac1{N^2} \sum_{1\leq i\neq j\leq N} \big[c_0 - G\big(x_i-x_j\big) \big],$$
where $c_0$ is a constant such that $G(x) \leq c_0,\, x\in \T^2$. We introduce the constant so that $\mathcal H_N$ is nonnegative. Recall that $X^N_t = \big( X^{N,1}_t, \ldots, X^{N,N}_t \big),\, t\geq 0 $ is the solution to the particle system \eqref{particle-systems}. We will show

\begin{lemma}\label{lem-Hamiltonian}
It holds that
  $$ \sup_{N\geq 2} \sup_{t\geq 0} \E \mathcal H_N\big(X^N_t \big) <+\infty.  $$
\end{lemma}

\begin{proof}
By the definition of $\mathcal H_N\big(X^N_t \big)$ and the exchangeability,
  $$\E \mathcal H_N\big(X^N_t \big) = \frac1{N^2} \sum_{1\leq i\neq j\leq N} \E \big[c_0 - G\big(X^{N,i}_t -X^{N,j}_t\big) \big] = \frac{N-1}{N} \E \big[c_0 - G\big(X^{N,1}_t -X^{N,2}_t\big) \big]. $$
Using the joint density function $F^{N,2}_t$ of $\big( X^{N,1}_t, X^{N,2}_t \big)$, we have
  $$\E \mathcal H_N\big(X^N_t \big) \leq \int_{\T^4} [c_0 -G(x-y)] F^{N,2}_t(x,y)\,\d x\d y .$$
Thanks to the formula \eqref{Green-funct} of the Green function $G$ on $\T^2$, we can find some big $c_1$ such that
  $$c_0 -G(x-y) \leq \log\frac{c_1}{|x-y|^{1/2\pi}}, \quad x,y\in \T^2, x\neq y.$$
Therefore,
  $$\aligned
  \E \mathcal H_N\big(X^N_t \big) &\leq \int_{\T^4} \Big( \log\frac{c_1}{|x-y|^{1/2\pi}} \Big) F^{N,2}_t(x,y)\,\d x\d y \\
  &\leq \log\bigg[ \int_{\T^4 } \exp\Big(\log \frac{c_1}{|x-y|^{1/2\pi}} \Big)\,\d x\d y \bigg] + \int_{\T^4 } \big( F^{N,2}_t\log F^{N,2}_t\big) (x,y) \,\d x\d y ,
  \endaligned $$
where the second step follows from Young's inequality (cf. \cite[Lemma 6.45]{Stroock} or \cite[Lemma 2.4]{ATW}) and the fact that $F^{N,2}_t$ is a probability density on $\T^4$. Note that the last integral is nothing but $2 h_2\big(F^{N,2}_t\big)$; by \eqref{entropy} and Lemma \ref{lem-entropy}, we deduce that
  $$\E \mathcal H_N\big(X^N_t \big)\leq \log\bigg[ \int_{\T^4 } \frac{c_1}{|x-y|^{1/2\pi}}\,\d x\d y \bigg]+ 2 h_2\big(F^{N,2}_t\big) \leq c_2 + 4h_1(f_0) $$
for some constant $c_2>0$. The above bound is independent of $t\geq 0$ and $N\geq 1$.
\end{proof}

We can deduce the following non-concentration result for point vortices. Let $B_r(x)$ be a ball with center $x\in \T^2$ and radius $r>0$.

\begin{corollary}\label{cor-no-concentrat}
It holds that
  $$\lim_{n\to \infty} \lim_{r\to 0} \sup_{N\geq n} \E\bigg[\sup_{x\in \T^2} S^N_t(B_r(x)) \bigg] =0. $$
In particular, $\P$-a.s., for all $t\in [0,T]$, the limit $\xi_t$ is a continuous measure on $\T^2$, i.e. it does not contain delta Dirac mass.
\end{corollary}

\begin{proof}
The second assertion follows from the first limit, as mentioned at the bottom of \cite[p.1086]{Schochet95}. To show the limit, we follow the idea in \cite[p.928, (3.5)]{Schochet96} which deals with the deterministic setting. Given small $r>0$, we have, for any $x\in \T^2$,
  $$\aligned
  \frac{\mathcal H_N\big(X^N_t \big)}{\log\frac1{2r}} &= \frac1{N^2} \sum_{1\leq i\neq j\leq N} \frac1{\log\frac1{2r}} \big[c_0 - G\big(X^{N,i}_t - X^{N,j}_t \big)\big]\\
  &\geq \frac1{N^2} \sum_{\substack{1\leq i\neq j\leq N\\ |X^{N,i}_t-x|\vee |X^{N,j}_t-x|<r}} \frac1{\log\frac1{2r}} \big[c_0 - G\big(X^{N,i}_t - X^{N,j}_t \big)\big],
  \endaligned $$
where $a\vee b =\max\{a,b\}$. Choosing a bigger $c_0$ if necessary, we can assume that
  \begin{equation}\label{cor-no-concentrat.1}
  c_0 - G(x-y) \geq \frac1{2\pi} \log\frac1{|x-y|} \quad \mbox{for all } x,y\in \T^2 ,\, 0<|x-y|\leq \frac12.
  \end{equation}
As a result,
  $$\aligned
  \frac{\mathcal H_N\big(X^N_t \big)}{\log\frac1{2r}} &\geq \frac1{2\pi} \frac1{N^2} \sum_{\substack{1\leq i\neq j\leq N\\ |X^{N,i}_t-x|\vee |X^{N,j}_t-x|<r}} \frac1{\log\frac1{2r}} \log\frac1{\big| X^{N,i}_t - X^{N,j}_t \big|} \\
  &\geq \frac1{2\pi} \sum_{\substack{1\leq i\neq j\leq N\\ |X^{N,i}_t-x|\vee |X^{N,j}_t-x|<r}} \frac1{N^2},
  \endaligned$$
where the second step follows from $\big| X^{N,i}_t - X^{N,j}_t \big|\leq \big| X^{N,i}_t-x \big| + \big|X^{N,j}_t-x \big|<2r$. Therefore,
  $$\aligned
  \frac{\mathcal H_N\big(X^N_t \big)}{\log\frac1{2r}} &\geq \frac1{2\pi} \Bigg[\bigg(\sum_{|X^{N,i}_t-x|<r} \frac1N \bigg)^2- \sum_{|X^{N,i}_t-x|<r} \frac1{N^2} \Bigg]  \\
  &\geq \frac1{2\pi} \bigg[\big(S^N_t(B_r(x)) \big)^2- \frac1{N} \bigg].
  \endaligned$$
This holds for any $x\in \T^2$, and thus we obtain
  $$\sup_{x\in \T^2} S^N_t(B_r(x)) \leq \frac1{\sqrt N} + \bigg(\frac{2\pi \mathcal H_N\big(X^N_t \big)}{\log\frac1{2r}}\bigg)^{1/2}. $$
Combining this inequality with Lemma \ref{lem-Hamiltonian} we immediately get the desired limit.
\end{proof}

We also have the following energy estimate on the weak limits.

\begin{corollary}\label{cor-regularity}
It holds that
  $$\sup_{t\in [0,T]} \E\<-G, \xi_t\otimes \xi_t\> < +\infty$$
and hence, $\P$-a.s., $\xi \in L^2\big(0,T; H^{-1}(\T^2)\big)$.
\end{corollary}

\begin{proof}
By the definition of the Hamiltonian,
  $$\mathcal H_N\big(X^N_t \big) = \frac{N-1}N c_0 + \big\<-G, S^N_t \otimes S^N_t \big\>,$$
where we have made the convention that $G(x-x)\equiv 0$ for all $x\in \T^2$. By Lemma \ref{lem-Hamiltonian}, there is $L>0$ such that
  $$\sup_{N\geq 2} \sup_{t\in [0,T]} \E \big\< -G, S^N_t \otimes S^N_t \big\> \leq L .$$
For any $\eps>0$, take $G_\eps \in C^\infty(\T^2)$ such that $G_\eps \geq G$ and $G_\eps (x) =G(x)$ for all $|x|\geq \eps$. Then,
  $$ \E \big\< -G_\eps, S^N_t \otimes S^N_t \big\> \leq \E \big\< -G, S^N_t \otimes S^N_t \big\> \leq L $$
for any $N\geq 2$ and $t\in [0,T]$. As $G_\eps$ is smooth on $\T^2$ we can apply the dominated convergence theorem to get $\E \< -G_\eps, \xi_t \otimes \xi_t \> \leq L$. By Fatou's lemma, letting $\eps \to 0$ yields
  \begin{equation}\label{cor-regularity.1}
  \E \< -G, \xi_t \otimes \xi_t \> \leq L  \quad \mbox{for all } t\in [0,T].
  \end{equation}

Next, denoting by $\psi_t = - G\ast \xi_t$ the stream function, then
  \begin{equation}\label{cor-regularity.2}
  \< -G, \xi_t \otimes \xi_t \> = \<\psi_t , \xi_t\>
  \end{equation}
where $\<\cdot, \cdot\>$ is now the duality between $\psi_t$ and $\xi_t$. If we use the Fourier expansion
  $$\xi_t = \sum_{k\in \Z^2_0} \<\xi_t, e_k\> e_k, $$
then we have
  $$\psi_t = \frac1{4\pi^2} \sum_{k\in \Z^2_0} \frac1{|k|^2} \<\xi_t, e_k\> e_k,$$
and thus
  $$\<\psi_t , \xi_t\> = \frac1{4\pi^2} \sum_{k\in \Z^2_0} \frac1{|k|^2} \<\xi_t, e_k\>^2= \frac1{4\pi^2} \|\xi_t \|_{H^{-1}}^2. $$
We obtain from \eqref{cor-regularity.1} and \eqref{cor-regularity.2} that
  $$\E \big(\|\xi_t \|_{H^{-1}}^2 \big)\leq 4\pi^2 L $$
for all $t\in [0,T]$. This concludes the proof.
\end{proof}

Now we proceed to the proof of convergence of the nonlinear term in \eqref{eq-new}; recall that we omit the tilde over $\tilde S^{N_i}$ and write $N$ instead of $N_i$. Take $\rho\in C^\infty(\R_+, [0,1])$ with support in $[0,1]$ and $\rho|_{[0,1/2]} \equiv 1$; for any $\delta\in (0,1)$, let $\rho_\delta(t) = \rho(t/\delta),\, t\in \R$. Then, we have the decomposition below:
  $$\big\< S^{N}_s\otimes S^{N}_s, H_\phi \big\> = \int_{\T^4} H_\phi(x,y)\, S^{N}_s(\d x) S^{N}_s(\d y) = I^N_1(s) + I^N_2(s), $$
where
  $$\aligned
  I^N_1(s)&= \int_{\T^4} H_\phi(x,y)(1- \rho_\delta(|x-y|)) S^{N}_s(\d x) S^{N}_s(\d y),\\
  I^N_2(s)&= \int_{\T^4} H_\phi(x,y) \rho_\delta(|x-y|) S^{N}_s(\d x) S^{N}_s(\d y).
  \endaligned $$
Note that $|H_\phi(x,y)| \leq C\|\nabla^2 \phi\|_\infty$ for all $x\neq y$, and $H_\phi(x,x) \equiv 0$ by convention; thus, $\P$-a.s., $I^N_1(s)$ and $I^N_2(s)$ are uniformly bounded by $C\|\nabla^2 \phi\|_\infty$ for all $s\in (0,T)$. Moreover, $H_\phi(x,y)(1- \rho_\delta(|x-y|))$ is smooth on $\T^4$; by the $\P$-a.s. convergence of $S^{N}_\cdot$ to $\xi_\cdot$ in the topology of $C\big([0,T]; H^{-r}(\T^2)\big)$ for all $r>1$, we have, $\P$-a.s., for all $s\in (0,T)$,
  $$\lim_{N\to \infty} I^N_1(s)= \int_{\T^4} H_\phi(x,y)(1- \rho_\delta(|x-y|))\, \xi_s(\d x) \xi_s(\d y). $$

Next we estimate $I^N_2(s)$:
  $$\aligned
  \big| I^N_2(s)\big| &\leq C\|\nabla^2 \phi\|_\infty \frac1{N^2} \sum_{1\leq i\neq j\leq N} \rho_\delta\big(\big|X^{N,i}_s -X^{N,j}_s \big|\big) \\
  &\leq C\|\nabla^2 \phi\|_\infty \frac1{N^2} \sum_{1\leq i\neq j\leq N} \rho_\delta\big(\big|X^{N,i}_s -X^{N,j}_s \big|\big) \frac1{\log\frac1\delta} \log\frac1{\big|X^{N,i}_s -X^{N,j}_s \big|} ,
  \endaligned $$
where the second step is due to $\rho_\delta\big(\big|X^{N,i}_s -X^{N,j}_s \big|\big)=0$ for $\big|X^{N,i}_s -X^{N,j}_s \big|>\delta$. Using \eqref{cor-no-concentrat.1} and the fact that $0\leq \rho_\delta \leq 1$ we obtain
  $$\aligned \big| I^N_2(s)\big| &\leq C\|\nabla^2 \phi\|_\infty \frac{2\pi}{\log \frac1\delta} \frac1{N^2} \sum_{1\leq i\neq j\leq N} \big[c_0 - G\big(X^{N,i}_s -X^{N,j}_s \big) \big] \\
  &= C\|\nabla^2 \phi\|_\infty \frac{2\pi}{\log \frac1\delta} \mathcal H_N\big(X^N_s \big).
  \endaligned $$
Combining this estimate with Lemma \ref{lem-Hamiltonian} gives us
  \begin{equation*}
  \E \big| I^N_2(s)\big|\leq C'\frac{\|\nabla^2 \phi\|_\infty}{\log \frac1\delta} \quad \mbox{uniformly in } s\in (0,T),\, N\geq 1.
  \end{equation*}
Summarizing the above arguments and applying Lebesgue's dominated convergence theorem, we conclude that
  $$\aligned
  \lim_{N\to \infty} &\, \E \bigg[\sup_{t\in [0,T]} \bigg|\int_0^t \big\< S^{N}_s\otimes S^{N}_s, H_\phi \big\>\,\d s - \int_0^t \big\< \xi_s\otimes \xi_s, H_\phi \big\>\,\d s \bigg|\bigg] \\
  \leq&\ C'T \frac{\|\nabla^2 \phi\|_\infty}{\log \frac1\delta}+ \E \int_0^T \bigg|\int_{\T^4} H_\phi(x,y) \rho_\delta(|x-y|) \, \xi_s(\d x) \xi_s(\d y) \bigg|\,\d s.
  \endaligned $$
By the second assertion of Corollary \ref{cor-no-concentrat}, the right hand side vanishes as $\delta\to 0$, and thus
  $$\lim_{N\to \infty} \E \bigg[\sup_{t\in [0,T]} \bigg|\int_0^t \big\< S^{N}_s\otimes S^{N}_s, H_\phi \big\>\,\d s - \int_0^t \big\< \xi_s\otimes \xi_s, H_\phi \big\>\,\d s \bigg|\bigg] =0. $$

Combining the above limit with Proposition \ref{key-prop}, we can finally let $i\to \infty$ in \eqref{eq-new} to yield, for all $t\in [0,T]$,
  $$\<\xi_t,\phi\> = \<f_0,\phi\> + \int_0^t \big\< \xi_s\otimes \xi_s, H_\phi \big\>\,\d s + \nu \int_0^t \big\< \xi_s, \Delta \phi \big\>\,\d s. $$
Thus $\xi_\cdot$ satisfies the weak vorticity formulation of the deterministic 2D Navier-Stokes equation with initial data $f_0$, and thus we complete the proof of Theorem \ref{main-thm}.

\bigskip

\noindent \textbf{Acknowledgements.} The second author is grateful to the financial supports of the National Key R\&D Program of China (No. 2020YFA0712700), the National Natural Science Foundation of China (Nos. 11931004, 12090014) and the Youth Innovation Promotion Association, CAS (Y2021002).

\end{document}